\documentclass[11pt]{article}
\usepackage{amsmath,amsthm,amsfonts,amssymb,bm,wasysym}
\usepackage{subfigure}
\usepackage{epsfig}
\usepackage[usenames]{color}
\usepackage{verbatim}

\parskip \medskipamount
\newtheorem{theorem}{Theorem}[section]
\newtheorem{definition}[theorem]{Definition}

\numberwithin{equation}{section}
\newtheorem{lemma}[theorem]{Lemma}
\newtheorem{proposition}[theorem]{Proposition}
\newtheorem{corollary}[theorem]{Corollary}

\newtheorem{remark}[theorem]{Remark}

\newtheorem{claim}{Claim}[section]

\numberwithin{equation}{section}

\def\N{\mathbb{N}}
\def\Z{\mathbb{Z}}

\def\R{\mathbb{R}}

\def\AA{\mathcal{A}}

\renewcommand{\phi}{\varphi}
\renewcommand{\epsilon}{\varepsilon}

\newcommand{\1}{{\text{\Large $\mathfrak 1$}}}

\newcommand{\tmix}{t_{\mathrm{mix}}}

\newcommand{\tl}{t_{\mathrm{L}}}
\newcommand{\thit}{t_{\mathrm{H}}}

\newcommand{\tct}{t_{\mathrm{cts}}}

\newcommand{\tmov}{t_{\mathrm{mov}}}

\newcommand{\pr}[1]{\mathbb{P}\!\left(#1\right)}
\newcommand{\E}[1]{\mathbb{E}\!\left[#1\right]}
\newcommand{\estart}[2]{\mathbb{E}_{#2}\!\left[#1\right]}
\newcommand{\prstart}[2]{\mathbb{P}_{#2}\!\left(#1\right)}
\newcommand{\prcond}[3]{\mathbb{P}_{#3}\!\left(#1\;\middle\vert\;#2\right)}

\newcommand{\escond}[3]{\mathbb{E}_{#3}\!\left[#1\;\middle\vert\;#2\right]}
 
\newcommand{\vol}[1]{\mathrm{vol}\!\left(#1\right)}

\newcommand\be{\begin{equation}}
\newcommand\ee{\end{equation}}

\begin{document}

\title{\bf Mixing times and moving targets}

\author{
Perla Sousi\thanks{University of Cambridge, Cambridge, UK;   p.sousi@statslab.cam.ac.uk} \and 
Peter Winkler\thanks{Dartmouth College, Hanover, NH; peter.winkler@dartmouth.edu}
}
\maketitle

\begin{abstract}
We consider irreducible Markov chains on a finite state space.  We show that the mixing time of any such chain is equivalent to the maximum,
over initial states $x$ and moving large sets $(A_s)_s$, of the hitting time of $(A_s)_s$ starting from $x$.
We prove that in the case of the $d$-dimensional torus the maximum hitting time of moving targets is equal
to the maximum hitting time of stationary targets.  Nevertheless, we construct a transitive graph where these two quantities are not equal,
resolving an open question of Aldous and Fill on a ``cat and mouse'' game.
\newline
MSC 2010 \emph{subject classifications.}
Primary   60J10; Secondary 60J27, 60G40.
\end{abstract}

\section{Introduction}
\label{sec:intro}


Mixing times and hitting times are fundamental notions for finite-state Markov chains.  Both have been widely studied (see, e.g., \cite{AldFill}
or \cite{LevPerWil} for background and numerous references) and a great variety of techniques have been developed to analyze them.

We begin by fixing some notation and reviewing previous work relating these two quantities.

Let $(X_t)_{t \geq 0}$ be an irreducible Markov chain on a finite state space with transition matrix $P$ and stationary distribution $\pi$. For $x,y$ in the state space we write
\[
P^t(x,y)=\prstart{X_t=y}{x},
\]
for the transition probability in $t$ steps.

Let $\displaystyle d(t) = \max_{x}\|P^t(x,\cdot) - \pi \|$, where $\|\mu-\nu\|$ stands for the total variation distance between the two probability measures $\mu$ and $\nu$. The total variation mixing is defined as follows:
\[
\tmix(\epsilon)=\min\{t \geq 0: d(t) \leq \epsilon\}.
\]
We use the convention that $\tmix = \tmix(1/4)$.

Before stating our first theorem, we introduce the maximum hitting time of ``big'' sets. Let $\alpha<1/2$, then we define
\[
\thit(\alpha)=\max_{x,A: \pi(A)\geq \alpha} \estart{\tau_A}{x},
\]
where $\tau_A$ stands for the first hitting time of the set $A$. 

We say that two real-valued functions $f$ and $g$ are {\em equivalent}, denoted $f \asymp g$, if there are universal positive constants
$c$ and $c'$ such that $cf \le g \le c'f$.  If the constants are allowed to depend on a parameter $\alpha$, we write $f \asymp_\alpha g$.

Aldous (1981) related mixing and hitting times by proving that  $\tct \asymp \max_{x,A}\pi(A) \estart{\tau_A}{x}$ for all reversible chains, where $\tct$ is the mixing time of the continuous time chain. In two independent recent papers by Imbuzeiro Oliveira~\cite{Imbuzeiro} and Peres and Sousi~\cite{PeresSousi} it was proved that for all reversible chains, if $\alpha<1/2$, then
\begin{align}\label{eq:previous}
\tl \asymp_\alpha \thit(\alpha),
\end{align}
where $\tl$ is the mixing time of the lazy version of the chain, i.e.\ the chain with transition matrix $\frac{P+I}{2}$.

Very recently, Griffiths, Kang, Imbuzeiro Oliveira and Patel~\cite{SimonRoberto} showed that $\thit(\alpha)\leq \thit(1/2)/\alpha$ for all $\alpha<1/2$. Hence this together with~\eqref{eq:previous} or with Aldous' result implies that for all reversible chains if $\alpha\leq 1/2$, then
\[
\tl \asymp_{\alpha} \thit(\alpha),
\]
with the equivalence failing if $\alpha>1/2$.

For non-reversible chains equation~\eqref{eq:previous} may fail, e.g.\ for biased random walk on the cycle $\Z_n$ we have
$\tl \asymp n^2$, while $\thit(\alpha) \asymp n$, for any $\alpha>0$. During a lecture on~\cite{PeresSousi} by Yuval Peres, Guy Kindler proposed that for non-reversible chains the right analogue of~\eqref{eq:previous} involves moving targets. Our first result establishes this equivalence.

Let $\alpha \in (0,1)$ and $\AA(\alpha)$ denote the collection of sequences of sets defined as follows:
\[
\AA(\alpha) = \{ A = (A_t)_{t\geq 0}: \ \forall t\geq 0, \ \pi(A_t) \geq \alpha\}.
\]
For $A \in \AA(\alpha)$ define $\tau_A = \inf\{ t\geq 0: X_t \in A_t\}$ and 
\[
\tmov(\alpha) = \sup_{x,A \in \AA(\alpha)} \estart{\tau_A}{x}.
\]

\begin{theorem}\label{thm:movingtargets}
For $\alpha < 1/2$, $\tmix \asymp \tmov(\alpha)$.
\end{theorem}

We will prove Theorem~\ref{thm:movingtargets} in Section~\ref{sec:movingsets}.

\begin{remark}\rm{
We note that Theorem~\ref{thm:movingtargets} does not require the chain to be either lazy or reversible, as is the case for~\eqref{eq:previous}. In this setting the equivalence holds for any chain.
}
\end{remark}

Theorem~\ref{thm:movingtargets} and~\eqref{eq:previous} immediately give that for all reversible lazy chains and for any $\alpha<1/2$ 
\[
\tmov(\alpha) \asymp_\alpha \thit(\alpha).
\]
If the chain is not reversible, though, the above equivalence can fail. For instance, for the biased random walk on $\Z_n$,
if $A = (A_i)_i$ are sets moving at the same speed as the random walk, then $\E{\tau_A} \asymp n^2$ agreeing with the mixing time $\tl$.


We next consider the problem of colliding with a moving target on a graph. In the following theorem we show that in the case of toroidal grids, the best strategy for the target, to avoid collision as long as possible,
is to stay in place at the maximum distance from the starting point.
As a corollary, we show that in the 1-dimensional case the two quantities $\thit$ and $\tmov$ are equal.

\begin{theorem}\label{thm:randomwalkzn}
Let $X$ be a lazy simple random walk on $\Z_n^d$ and $f:\N \to \Z_n^d$ a function. Then setting $a = \left(\lfloor n/2 \rfloor,\ldots,\lfloor n/2 \rfloor\right)$ we have for all $t$
\[
\prstart{X_1 \neq f(1), \ldots, X_t \neq f(t)}{0} \leq \prstart{X_1 \neq a, \ldots, X_t \neq a}{0}.
\]
\end{theorem}

\begin{remark}
\rm{
Note that if the random walk $X$ on $\Z_n^d$ is not lazy, then one can always choose a function $f:\N\to \Z_n^d$ so that 
\[
\prstart{X_1\neq f(1),\ldots, X_t \neq f(t)}{0} =1,
\]
and hence the conclusion of Theorem~\ref{thm:randomwalkzn} fails.
}
\end{remark}

\begin{corollary}\label{pro:torus}
Let $X$ be a lazy simple random walk on $\Z_n= \{0,1,\ldots,n-1\}$. Then for all $n,\alpha$ we have
\[
\thit(\alpha) = \tmov(\alpha).
\]
\end{corollary}

We prove Theorem~\ref{thm:randomwalkzn} and  Corollary~\ref{pro:torus} in Section~\ref{sec:torus} using a discrete version of rearrangement inequalities. We employ a polarization technique which has been used extensively in the continuous setting to prove several classical rearrangement inequalities (see, for instance,~\cite{BurSch}). As a by-product of the discrete rearrangement inequality, we also prove that the expected volume of the ``sausage'' around a discrete lazy simple random walk on $\Z^d$ with drift is minimized when the drift is equal to $0$.


\begin{proposition}\label{pro:sausage}
Let $X$ be a lazy simple random walk on $\Z^d$ and let $f:\N \to \Z^d$ be a function. Then for all $t\in \N$ and all $n\in \N$
\[
\E{\vol{\bigcup_{s=0}^{t}(X_s + f(s) + Q_n)}} \geq \E{\vol{\bigcup_{s=0}^{t}(X_s +  Q_n)}},
\]
where $Q_n = [-n,n]^d$.
\end{proposition}

A more general isoperimetric inequality for the expected volume of the Wiener sausage has been proved in~\cite{PerSous11}; the stronger Proposition~\ref{pro:sausage} makes use of the symmetries of $\Z^d$ and does not hold in general.

Finally, in the last theorem, we show that that the equality of Proposition~\ref{pro:torus} is not always true for a reversible Markov chain.
This resolves an open question of Aldous~\cite[Chapter~4, Open Problem~20]{AldFill} and of Imbuzeiro Oliveira~\cite{Oliveira}.

We say that $X$ is a continuous time random walk on a graph if it stays at every vertex for an exponential amount of time of mean $1$,
and then jumps to one of the neighbours uniformly at random. 

\begin{theorem}\label{thm:example}
There exists a transitive graph $G=(V,E)$ such that if $X$ is a continuous time or lazy random walk on $G$, then 
\[
\max_{x,y}\estart{\tau_y}{x} < \sup_{x,f \in V^{\R_+}}\estart{\tau_f}{x},
\]
where $\tau_f = \inf\{t\geq 0: X_t = f(t)\}$.
\end{theorem}

In~\cite{AldFill} and~\cite{Oliveira} this was stated as a cat and mouse problem and it was conjectured that the best strategy
for the mouse to maximize the expected capture time is to stay in place. In our graph $G$ we show that this is not the case.
We prove Theorem~\ref{thm:example} in Section~\ref{sec:example}.

\section{Moving targets}
\label{sec:movingsets}

\begin{proof}[{\bf Proof of Theorem~\ref{thm:movingtargets}}]
We first show that $\tmov \leq c_1 \tmix$, where $c_1$ is a positive constant.

Let $t = \tmix(\alpha/2) \leq \lceil\log_2(1/\alpha)\rceil\tmix$. Then for all $x$ and all sets $A$ we have
\[
P^{t}(x,A) \geq \pi(A) - \frac{\alpha}{2}.
\]
Take a sequence of sets $A=(A_s) \in \AA(\alpha)$. Then for all $s$ and all starting points $x$ we have
\begin{align}\label{eq:allstarting}
P^t(x,A_s) \geq \frac{\alpha}{2}.
\end{align}
If $\tau= \min\{ k\geq 0: X_{kt} \in A_{kt}\}$, then obviously we have $\tau_A \leq t \tau$.
By~\eqref{eq:allstarting}, it follows that $\tau$ is stochastically dominated by a geometric random variable of success probability $\alpha/2$. Therefore,
\[
\estart{\tau_A}{x} \leq t \estart{\tau}{x} \leq \frac{2 t}{\alpha},
\]
and hence this gives that 
\[
\tmov \leq  \frac{2\lceil\log_2(1/\alpha)\rceil}{\alpha}\tmix
\]
and this completes the proof of the upper bound.

We now show the other direction, i.e.\ that there exists a positive constant $c_2$ so that 
\[
\tmix \leq c _2\tmov(\alpha).
\]
Since $\alpha<1/2$, there exists $\epsilon>0$ such that 
$\alpha+\epsilon <1/2$. By~\cite[4.35]{LevPerWil}, it follows that there exists a positive constant $c_3$ such that $\tmix(\alpha+ \epsilon) \geq c_3 \tmix$. Let $t<\tmix(\alpha+\epsilon)$. Then this means that there exists $x$ and a set $A$ so that 
\begin{align}\label{eq:badpoint}
P^{t}(x,A) <\pi(A) - (\alpha+\epsilon).
\end{align}
From that we immediately get that $\pi(A) >\alpha+\epsilon$. We now use the set $A$ to define a sequence of sets $(B_s)$ as follows: for $s < t$ define
\[
B_s = \{y: P^{t-s}(y,A) > \pi(A) - \alpha\}
\]
and for $s\geq t$ we let $B_s = \Omega$. 
Since $\pi$ is stationary, it follows that
\[
\pi(A) = \sum_{y\in B_s} P^{t-s}(y,A)\pi(y) + \sum_{y\in B_s^c} P^{t-s}(y,A)\pi(y) \leq \pi(B_s) + \pi(A) - \alpha.
\]
Rearranging, gives that 
$\pi(B_s) \geq \alpha$ for all $s$. 
We write $\tau_B = \min\{ t\geq 0: X_t \in B_t\}$. We will show  for a constant $\theta$ to be determined later we have that  
\begin{align}\label{eq:contradtrue}
\estart{\tau_B}{x} \geq \theta t.
\end{align}
We will show that for a $\theta$ to be specified later, assuming 
\begin{align}\label{eq:contrad}
\max_z \estart{\tau_B}{z} \leq \theta t
\end{align}
will yield a contradiction. 

By Markov's inequality and~\eqref{eq:contrad} we have that for all $z$
\[
\prstart{\tau_B \leq t}{z} \geq 1 - \theta.
\]
By the strong Markov property applied to the stopping time $\tau_B$ and Markov's inequality we have
\[
\prstart{X_t \in A}{x} \geq \prcond{X_t \in A}{\tau_B\leq t}{x} \prstart{\tau_B \leq t}{x} \geq \inf_{s\leq t} \inf_{w \in B_s} \prstart{X_{t-s} \in A}{w} (1-\theta) \geq 
(\pi(A) - \alpha)(1-\theta),
\]
which by choosing $\theta$ small enough can be made bigger than $\pi(A) - (\alpha + \epsilon)$. This contradicts the choice of  $x$ in~\eqref{eq:badpoint}. Therefore~\eqref{eq:contradtrue} holds and this completes the proof.
\end{proof}

\section{Collision with a moving target on $\Z_n^d$ and $\Z^d$}
\label{sec:torus}

In this section we prove Theorem~\ref{thm:randomwalkzn} and Proposition~\ref{pro:sausage}. We start by introducing some notation and background on rearrangement inequalities. We follow closely~Section~2.1 of Burchard and  Schmuckenschl{\"a}ger~\cite{BurSch}.

\subsection{Notation and background}

Let $M$ be a metric space. A \textit{reflection} $\sigma:M\to M$  is an isometry such that
\begin{itemize}
\item $\sigma^2 x = x$ for all $x\in M$;

\item $M$ is the disjoint union of the set of fixed points $H^0$, and two half spaces $H^-$ and $H^+$ which are exchanged by $\sigma$, i.e.
\begin{align*}
&\sigma x = x \ \ x \in H^0, \\
& \sigma H^+ = H^-;
\end{align*}

\item $d(x,y) < d(x,\sigma y)$ for all $x,y \in H^+$.
\end{itemize}

From now on whenever we define a reflection $\sigma$ we will specify
$H^+$ and $H^-$.

The \textit{two-point rearrangement} of a function $f$ is defined to be
\[
f^{\sigma}(x) = \left\{
                           \begin{array}{ll}
                             \max\{f(x),f(\sigma x)\}, & \hbox{if  $x \in H^+$;} \\
                            \min\{f(x),f(\sigma x)\}, & \hbox{if $x \in H^-$;}\\
                             f(x), & \hbox{if $x \in H^0$.}
                           \end{array}
                         \right.
\]
By taking $f = \1(A)$ we get that the two-point rearrangement of a set $A$, denoted $A^\sigma$, satisfies
\begin{align*}
& A^\sigma \cap H^+ = (A \cup \sigma A) \cap H^+ 
\\
& A^{\sigma} \cap H^- = (A \cap \sigma A) \cap H^-.
\end{align*}

We now recall a combinatorial lemma from~\cite[Lemma~2.6]{BurSch}.

Consider the two-point space  $\{+,-\}$ with the metric defined by $d(+,-)=1$. The map $\sigma$ that exchanges $+$ and $-$ is a reflection with no fixed points and with $H^+ = \{+\}$ and $H^-=\{-\}$ as the positive and negative half-spaces. For any function $\phi$ on $\{+,-\}$, let $\phi^\sigma$ be the corresponding two-point rearrangement of $\phi$:
\begin{align}\label{eq:almutsigma}
\phi^{\sigma}(+) = \max\{ \phi(+),\phi(-)\} \ \ \text{ and } \ \ \phi^\sigma(-) = \min\{ \phi(+),\phi(-)\}.
\end{align}

\begin{lemma}[Burchard and  Schmuckenschl{\"a}ger~\cite{BurSch}]\label{lem:almut}
Let  $\phi_1,\ldots,\phi_n$ be nonnegative functions on the set $\{+,-\}$. For each pair $ij$, let $k_{i,j}(\epsilon,\epsilon') = a_{ij} + b_{ij} \1(\epsilon = \epsilon')$ with $a_{ij} , b_{ij} \geq 0$. Consider the function 
\[
J(\phi_1,\ldots,\phi_n) = \sum_{\pm} \prod_{1\leq i \leq n} \phi_i(\epsilon_i) \prod_{1\leq i \leq j\leq n} k_{i,j}(\epsilon_i,\epsilon_j).
\]
Then
\[
J(\phi_1,\ldots,\phi_n) \leq  J(\phi_1^{\sigma},\ldots,\phi_n^\sigma).
\]
\end{lemma}

\subsection{Random walk on $\Z_n^d$}

\begin{lemma}\label{lem:startata}
Let $\sigma$ be a reflection in $\Z_n^d$ and $X$ a lazy simple random walk in $\Z_n^d$. Then for all times $t$, all starting states $b$ and all sets $D_i\subseteq \Z_n^d$ we have
\[
\prcond{X_1 \in D_1,\ldots, X_t \in D_t}{X_0 \in \{b\}}{} \leq \prcond{X_1 \in D_1^\sigma,\ldots, X_t \in D_t^\sigma}{X_0 \in \{b\}^{\sigma}}{}.
\]
\end{lemma}

\begin{proof}[{\bf Proof}]

Let $p(x,y)$ be the transition probability in one step of the lazy simple random walk in $\Z_n^d$, i.e.\
\[
p(x,y) =  \1(x=y) \frac{1}{2} + \1(|x-y| =1)\frac{1}{4d}.
\]
By the Markov property we have
\[
\prcond{X_1 \in D_1,\ldots, X_t \in D_t}{X_0 \in \{b\}}{} = \sum_{x_0,\ldots, x_t} \prod_{i=1}^{t} p(x_{i-1},x_i) \prod_{i=0}^{t} \1(x_i \in D_i),
\]
where $D_0 = \{b\}$.
If $H^+$ and $H^-$ are the positive and negative respectively half spaces exchanged by $\sigma$, then we can write the above sum 
\[
\sum_{x_0,\ldots, x_t} \prod_{i=1}^{t} p(x_{i-1},x_i) \prod_{i=0}^{t} \1(x_i \in D_i) = \sum_{x_0,\ldots, x_t \in H^+} \sum_{\pm} \prod_{i=1}^{t} p(x_{i-1}^{\pm}, x_i^{\pm}) \prod_{i=0}^{t} \1(x_i^{\pm} \in D_i),
\]
where 
\begin{align}\label{eq:defplus}
x^+ = \left\{
                           \begin{array}{ll}
                             x, & \hbox{if  $x \in H^+$} \\
                           \sigma x, & \hbox{if $x \in H^-$}
                           \end{array}
                         \right.
                        \ \text{ and } \
x^-= \left\{
                           \begin{array}{ll}
                             \sigma x, & \hbox{if  $x \in H^+$} \\
                              x, & \hbox{if $x \in H^-$.}
                           \end{array}
                         \right.
                         \end{align}  
We now fix a choice of $x_1,\ldots, x_t \in H^+$. It suffices to show that 
\begin{align}\label{eq:goalrear}
\sum_{\pm} \prod_{i=1}^{t} p(x_{i-1}^{\pm}, x_i^{\pm}) \prod_{i=0}^{t} \1(x_i^{\pm} \in D_i) \leq \sum_{\pm} \prod_{i=1}^{t} p(x_{i-1}^{\pm}, x_i^{\pm}) \prod_{i=0}^{t} \1(x_i^{\pm} \in D_i^\sigma).
\end{align}                              
For $\epsilon,\epsilon' \in \{+,-\}$ we define $k_{i,j}(\epsilon,\epsilon') =1$ if $j-i \neq 1$ and otherwise
\[
k_{i-1,i}(\epsilon,\epsilon') = p(x_{i-1}^-, x_i^+) + \1(\epsilon = \epsilon') (p(x_{i-1}^+,x_i^+) - p(x_{i-1}^-,x_i^+)).
\]
By the definition of the transition probability we have $p(x_{i-1}^-, x_i^+) \leq p(x_{i-1}^+,x_i^+)$ for $x_{i-1},x_i \in H^+$. Therefore $k_{i,j}$ satisfies the assumptions of Lemma~\ref{lem:almut} and if we set $\phi_i(\epsilon) = \1(x_i^{\epsilon} \in D_i)$, then
we can write
\begin{align}\label{eq:firsteq}
\sum_{\pm} \prod_{i=1}^{t} p(x_{i-1}^{\pm}, x_i^{\pm}) \prod_{i=0}^{t} \1(x_i^{\pm} \in D_i) = \sum_{\pm} 
\prod_{i=0}^{t} \phi_i(\epsilon_i)
\prod_{0\leq i\leq j\leq t} 
k_{i,j}(\epsilon_{i},\epsilon_j).
\end{align}
Applying Lemma~\ref{lem:almut} we infer
\begin{align}\label{eq:application}
\sum_{\pm} 
\prod_{i=0}^{t} \phi_i(\epsilon_i)
\prod_{0\leq i\leq j\leq t} 
k_{i,j}(\epsilon_{i},\epsilon_j) \leq \sum_{\pm} 
\prod_{i=0}^{t} \phi_i^\sigma(\epsilon_i)
\prod_{0\leq i\leq j\leq t} 
k_{i,j}(\epsilon_{i},\epsilon_j).
\end{align}
Since $\phi_i^{\sigma}(\epsilon) = \1(x_i^\epsilon \in D_i^{\sigma})$, inequality~\eqref{eq:application} together with~\eqref{eq:firsteq} concludes the proof of~\eqref{eq:goalrear} and thus completes the proof of the lemma.
\end{proof}

\begin{remark}\rm{
Note that it is essential that the random walk on $\Z_n^d$ be lazy. In the proof above this was used to show that the kernel $k$ satisfies the assumptions of Lemma~\ref{lem:almut}. 
}
\end{remark}

\begin{figure}
\begin{center}
\epsfig{file=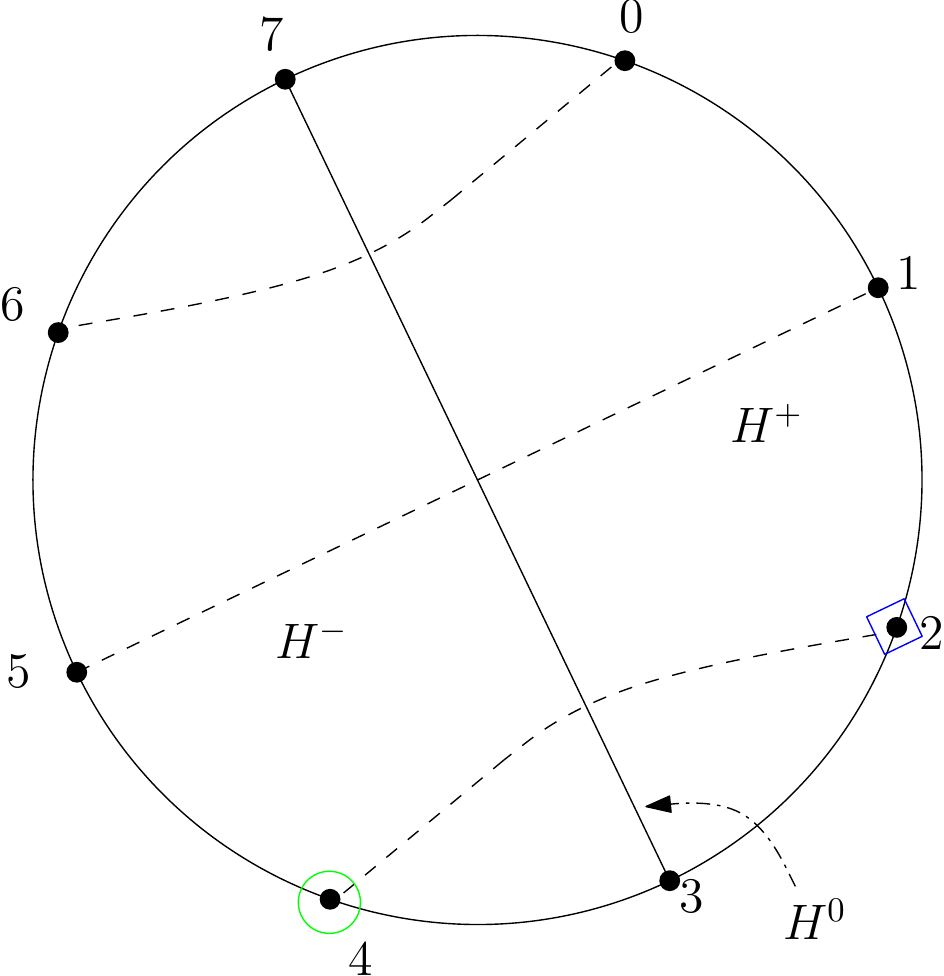, scale=0.7}
\caption{\label{fig:gross} A reflection on $\Z_8$}
\end{center}
\end{figure}

\begin{proof}[{\bf Proof of Theorem~\ref{thm:randomwalkzn}}]

We first prove the theorem for $d=1$. 

For $i=1,\ldots, t$ we write $D_i = \Z_n \setminus \{f(i)\}$.
Then we have
\[
\prstart{X_1 \neq f(1), \ldots, X_t \neq f(t)}{0} = \prstart{X_1 \in D_1, \ldots, X_t \in D_t}{0}.
\]
We now want to find a sequence of reflections $\sigma_1,\ldots, \sigma_k$ such that $D_i^{\sigma_1\ldots \sigma_k} = \Z_n \setminus \{ a\}$. 

We first give the reflection $\sigma$ such that $D_1^{\sigma} = \Z_n \setminus \{a\}$.
We carry out all the details in the case when $n$ is odd and $f(1) + a$ is even and satisfies $f(1)+a \geq n-1$.
The other cases follow similarly. We define
\[
\sigma_1(x) = (a+ f(1) - x) \bmod n 
\]
and we let 
\begin{align*}
&H^+ = \Z_n \cap \left(\left( \frac{a+f(1)}{2} , n-1\right] \cup \left[ 0, \frac{a+f(1)}{2} - \frac{n-1}{2} \right)   \right) \ \text{ and }
\\
& H^- = \left(H^+\right)^c \setminus \left\{\frac{a+f(1)}{2}\right\}.
\end{align*}

Then with this definition of $H^+$ and $H^-$ it is clear that $D_1^{\sigma_1} = \Z_n\setminus\{a\}$ and $\left(\Z_n \setminus\{a\}\right)^{\sigma_1} = \Z_n\setminus\{a\}$ and $\{0\}^{\sigma_1} = \{0\}$.

Having symmetrized the set $D_1$, we now want to find a reflection $\sigma_2$ such that $D_2^{\sigma_1\sigma_2} = \Z_n \setminus \{a\}$. 
To do that we use exactly the same construction as for $\sigma_1$ above. Hence we get $\left(\Z_n\setminus\{a\}\right)^{\sigma_2} =\Z_n\setminus \{a\}$ and $\{0\}^{\sigma_2} = \{0\}$. Therefore $D_1^{\sigma_1 \sigma_2}=\Z_n\setminus \{a\}$. Continuing in this manner we find $k\leq t$ reflections $\sigma_1,\ldots, \sigma_k$ such that for all $i$
\[
D_i^{\sigma_1\ldots \sigma_k} = \Z_n\setminus \{a\}.
\]

\begin{figure}
\begin{center}
\epsfig{file=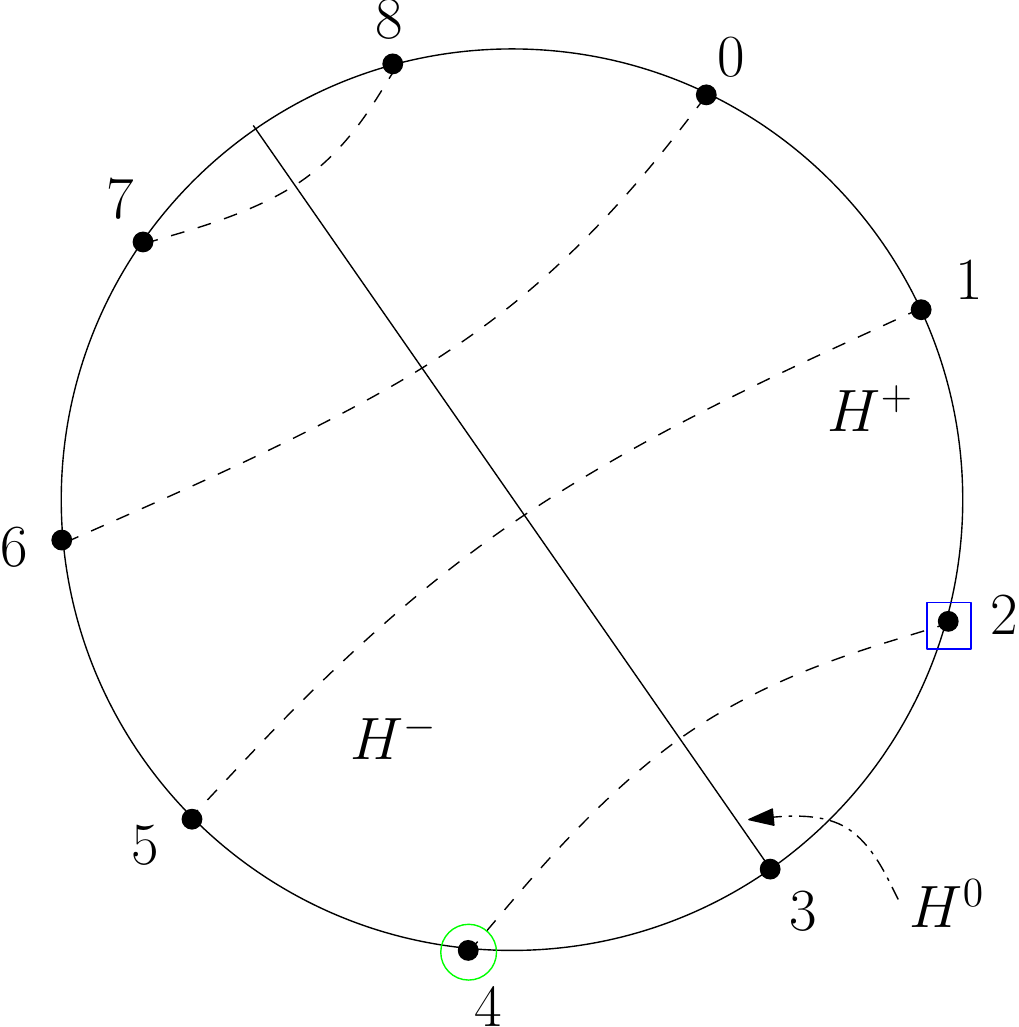, scale=0.7}
\caption{\label{fig:gross} A reflection on $\Z_9$}
\end{center}
\end{figure}

Applying Lemma~\ref{lem:startata} $k$ times when $b=0$, i.e.\ for the reflections $\sigma_1,\ldots,\sigma_k$, concludes the proof in the case $d=1$. 

For higher dimensions, the statement follows from carrying out the above procedure coordinate by coordinate. 
\end{proof}

\begin{remark}\rm{
We note that for a continuous time random walk on $\Z_n^d$ 
the analogue of Theorem~\ref{thm:randomwalkzn} holds, i.e.\ 
\[
\prstart{X_s \neq f(s), \forall s\leq t}{0} \leq \prstart{X_s \neq a, 
\forall s\leq t}{0}.
\]
To see this, view the continuous time walk as the continuous time version of the lazy walk with exponential clocks of rate $2$. Then condition on the number of lazy steps taken by the continuous time walk by time $t$, and apply Theorem~\ref{thm:randomwalkzn}. 
}
\end{remark}

\begin{proof}[{\bf Proof of Proposition~\ref{pro:torus}}]

It is clear that on the cycle $\Z_n$ among all sets $A$ of the same measure the hardest to hit is an interval. The first hitting time of an interval on the cycle is the same as the first hitting time of the endpoints, which can be glued to a single point, and hence the hitting time is maximized when this point is staying fixed.
\end{proof}

\subsection{Random walk on $\Z^d$}

In this section we prove Proposition~\ref{pro:sausage}. The proof follows in a similar way to the proof of Proposition~\ref{pro:torus} and uses again Lemma~\ref{lem:almut}.

\begin{lemma}\label{lem:prelimvol}
Let $X$ be a lazy simple random walk on $\Z^d$ starting from $0$ and let $(D_i)$ be subsets of $\Z^d$ that are symmetric around the origin, i.e.\ $D_i = -D_i$ for all $i$. If $\sigma$ is a reflection on $\Z^d$, then for all $t$ we have
\[
\E{\vol{\bigcup_{s=0}^{t}\left(X_s + D_s\right)}} \geq \E{\vol{\bigcup_{s=0}^{t}\left(X_s + D_s^{\sigma} \right)}}.
\]
\end{lemma}

\begin{proof}[{\bf Proof}]

Since $D_s = -D_s$ for all $s$, we have
\begin{align*}
 \E{\vol{\bigcup_{s=0}^{t} ( X_s + D_s)}} &= \E{\sum_{x_0\in \Z^d} \1\left(x_0 \in \bigcup_{s=0}^{t}(X_s +D_s)\right)} \\
&= \E{\sum_{x_0\in \Z^d}\1\left(\exists s\leq t: \ X_s \in x_0 +D_s\right)}.  
\end{align*}
Let $p(x,y)$ be the transition probability in one step of the lazy simple random walk in $\Z^d$, i.e.\ 
\[
p(x,y) =  \1(x=y) \frac{1}{2} + \1(|x-y| =1)\frac{1}{4d}.
\]
Then the Markov property of the random walk gives
\begin{align} \label{eq:markov}
\pr{\exists s\leq t: \ X_s \in x_0 + D_s} = 1 - \sum_{y_1,\ldots, y_t} \prod_{i=1}^{t} p(y_{i-1},y_i) \prod_{i=0}^{t} \1(y_i \notin x_0 + D_i),
\end{align}
where $y_0 = 0$.
Changing variables to $y_i-x_0$ and noticing that $p(0,y_1 +x_0) = p(-x_0,y_1)$ gives that the sum 
appearing in the right-hand side of~\eqref{eq:markov}
is equal to
\begin{align*}
\sum_{y_1,\ldots,y_t} p(-x_0,y_1)\1(-x_0 \notin D_0) \prod_{i=2}^{t} p(y_{i-1},y_i) \prod_{i=1}^{t} \1(y_i \notin D_i).
\end{align*}
Putting everything together in the expression for the expected volume of $\cup_{s\leq t}(\xi(s) + D_s)$ and changing variables from $-x_0$ to $x_0$ we get
\begin{align}\label{eq:expressionvol}
\E{\vol{\bigcup_{s=0}^{t} Q_n(f(s) + \xi(s))}}  = 
\sum_{x_0,x_1,\ldots, x_t} \prod_{i=1}^{t} p(x_{i-1},x_i) \left( 1 - \prod_{i=0}^{t} \1(x_i \notin D_i )    \right).
\end{align}
Decomposing the above sum into the positive and negative half spaces of $\sigma$, the right hand side of~\eqref{eq:expressionvol} can be written as 
\[
\sum_{x_0,x_1,\ldots, x_t \in H^+} \sum_{\pm} \prod_{i=1}^{t} p(x_{i-1}^{\pm}, x_i^{\pm}) \left( 1 - \prod_{i=0}^{t} \1(x_i^{\pm} \notin D_i )    \right),
\]
where $x^{+}$ and $x^-$ are as defined in~\eqref{eq:defplus} in the proof of Lemma~\ref{lem:startata}. Repeating the same arguments as in the proof of~\eqref{eq:goalrear} in Lemma~\ref{lem:startata} we get 
\begin{align*}
\sum_{\pm} \prod_{i=1}^{t}
p(x_{i-1}^{\pm}, x_i^{\pm}) \left( 1 - \prod_{i=0}^{t} \1(x_i^{\pm} \notin D_i )    \right) \geq \sum_{\pm} \prod_{i=1}^{t}
p(x_{i-1}^{\pm}, x_i^{\pm}) \left( 1 - \prod_{i=0}^{t} \1(x_i^{\pm} \notin D_i^{\sigma})\right).
\end{align*}
Hence, we conclude that
\[
\E{\vol{\bigcup_{s=0}^{t} \left(\xi(s) + D_s\right)}} \geq \E{\vol{\bigcup_{s=0}^{t} \left(\xi(s) + D_s^\sigma \right)}}
\]
and this finishes the proof of the lemma.
\end{proof}

\begin{proof}[{\bf Proof of Proposition~\ref{pro:sausage}}]

Let $r>0,x\in \Z^d$ and $Q_r(x) = [-r+x_1,r+x_1] \times \ldots \times [-r + x_d, r + x_d]$ be the box in $\Z^d$ of side length $2r+1$ centered at $x$. We want to show that 
\[
\E{\vol{\bigcup_{s=0}^{t} \left( \xi(s) + Q_n(f(s))\right)  }} \geq \E{\vol{\bigcup_{s=0}^{t} \left(\xi(s) + Q_n\right)  }},
\]
where $Q_n = [-n,n]^d$ as defined in the statement of the proposition.

We now want to find a sequence of reflections $\sigma_1,\ldots, \sigma_k$ such that $Q_n(f(s))^{\sigma_1\ldots \sigma_k} = Q_n$ for all~$s \leq t$.

First we show how to bring a non-centered interval to a centered one in $\Z$. Let $A = [a-n,a+n]$, where $a\in \Z$ and $n\in \N$. Define the reflection $\sigma$ around the point $a/2$ via
\[
\sigma(x) = a - x.
\]
Then it is clear that $\sigma$ maps the interval $[a-n,a+n]$ to the interval $[-n,n]$. If $a>0$, define $H^+ = \{ k \in \Z: k\leq a/2\}$ and $H^-$ to be its complement. If $a<0$, define $H^+ = \{ k \in \Z: k\geq a/2\}$. It is then easy to see that $A^\sigma = [-n,n]$ and $[-n,n]^\sigma = [-n,n]$.

Next we define reflections in $\Z^d$. Let $A=[a_1-n_1,a_1 + n_1] \times \ldots \times [a_d-n_d,a_d + n_d]$ and for $i=1,\ldots, d$ let 
\[
\sigma_i(x_1,\ldots,x_d) = (x_1,\ldots,x_{i-1}, a_i - x_i, x_{i+1},\ldots, x_d ).
\]
Then $A^{\sigma_1\ldots \sigma_d} = [-n_1,n_1]\times \ldots \times [-n_d, n_d]$ and if $B$ is a centered rectangle, then $B^{\sigma_1\ldots \sigma_d} = B$.

This way we see that there exist $k\leq td$ reflections $\sigma_1,\ldots, \sigma_{k}$ such that 
\[
Q_n(f(s))^{\sigma_1\ldots \sigma_{k}} = Q_n \ \ \text{ for all } \ \ s\leq t.
\]
Applying Lemma~\ref{lem:prelimvol} $k$ times concludes the proof of the proposition.
\end{proof}

\section{Better to run than hide}
\label{sec:example}

\begin{figure}[h!]
\begin{center}
\epsfig{file=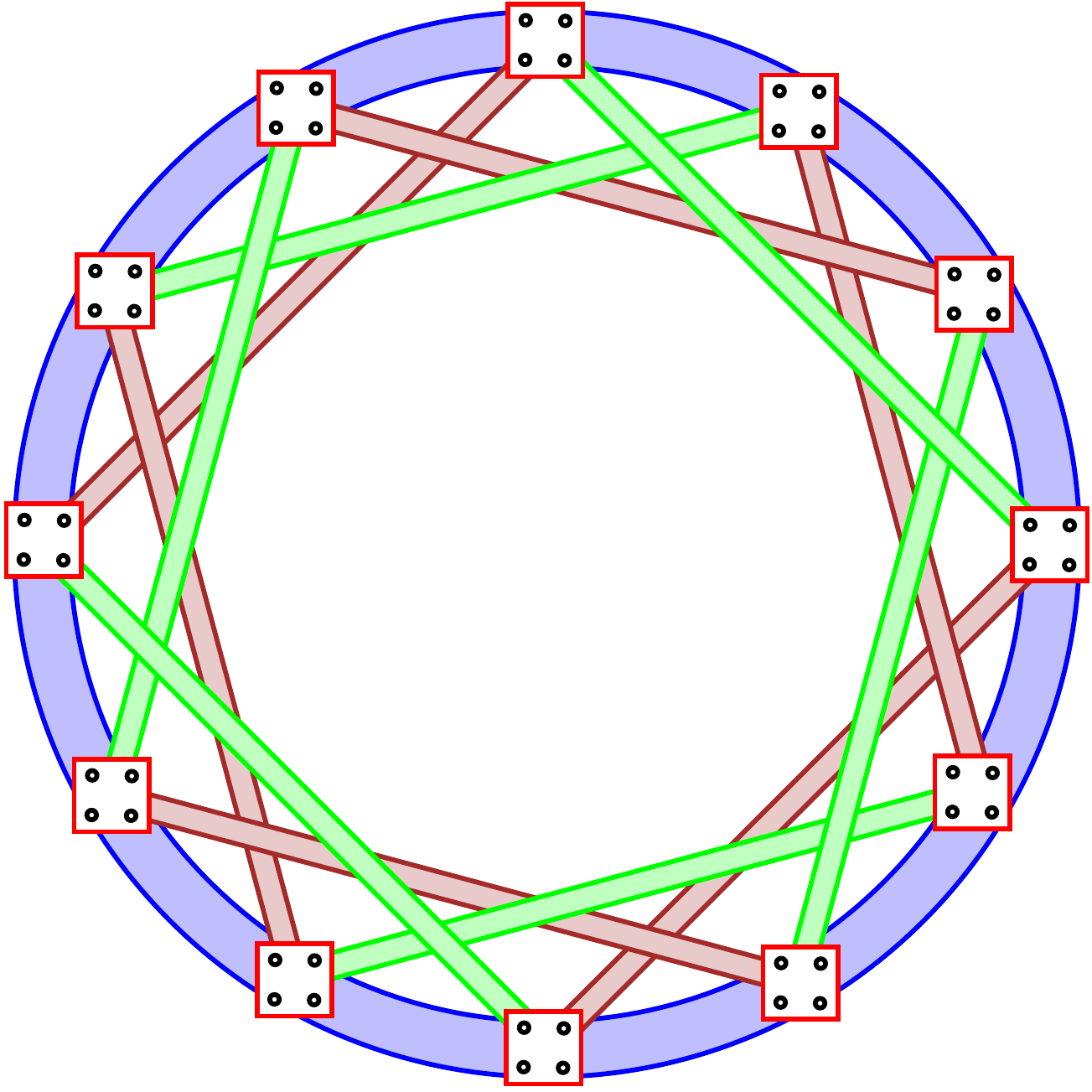, height=9cm}
\caption{\label{fig:gross} Graph $G_{2,12}$}
\end{center}
\end{figure}

In this section we give the proof of Theorem~\ref{thm:example}. We first define a class of graphs indexed by $n,m$ and denoted $G_{n,m}$. For $n=2$ and $m=12$ the graph is illustrated in Figure~\ref{fig:gross}. We then prove that $G_{2,12}$ is an example of a graph satisfying the statement of Theorem~\ref{thm:example} for a lazy discrete time walk. We conclude the section by proving that $G_{7,20}$ is such that it is best for a target to move in order to avoid collision with a continuous time walk.

\begin{definition}\rm{
Let $m$ be a multiple of $4$ and $G_{n,m}$ a graph on 
$n^2 m$ vertices divided into $m$ clusters. We think of the clusters as the nodes of $\Z_{m}$ and so we number them $0,\ldots,m{-}1$. We give coordinates to each element of every cluster. The elements of cluster $i$ have coordinates $i(a,b)$, where $a,b\in \Z_n$. 
We put an edge between
\begin{itemize}
\item[(1)] all pairs $i(a,b), j(c,d)$ with $|i-j|=1$;

\item[(2)] all pairs $i(a,b), j(a,d)$ with $b\neq d$, $i$ even and $j=(i+m/4)\bmod m$;

\item[(3)] all pairs $i(a,b), j(c,b)$ with $a\neq c$,  $i$ even and $j=(i-m/4)\bmod m$;

\item[(4)] all pairs $i(a,b),j(c,b)$ with $a\neq c$, $i$ odd and $j=(i+m/4)\bmod m$;

\item[(5)] all pairs $i(a,b), j(a,d)$ with $b\neq d$, $i$ odd and $j=(i-m/4)\bmod m$.

\end{itemize}

We call the edges of type (1) ``short'' while the edges of type (2), (3), (4) and (5) ``long''.
}
\end{definition}

\begin{remark}\rm{
Intuitively, notice that for a fixed $m$ as $n$ goes to infinity, the long edges of $G_{n,m}$ are rarely used, and hence $G_{n,m}$ looks more like $\Z_{m}$. 
}
\end{remark}

%
%
%
%
%
%
%

\begin{figure}
\label{fig:zoom}
\begin{center}
\subfigure[Short edges]{\label{fig:short}
\epsfig{file=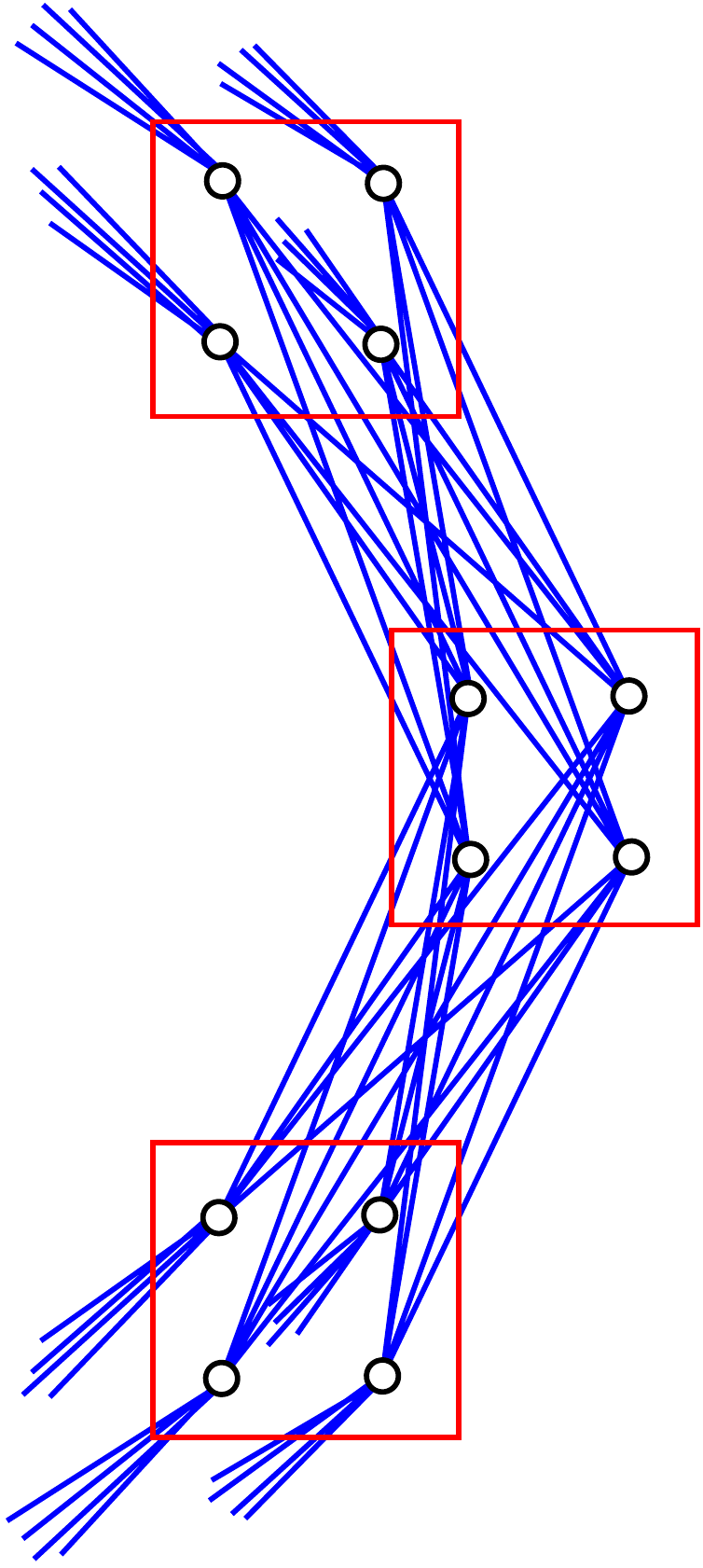, height=7cm}
}
\hspace{3cm}
\subfigure[Long edges]{\label{fig:long}
\epsfig{file=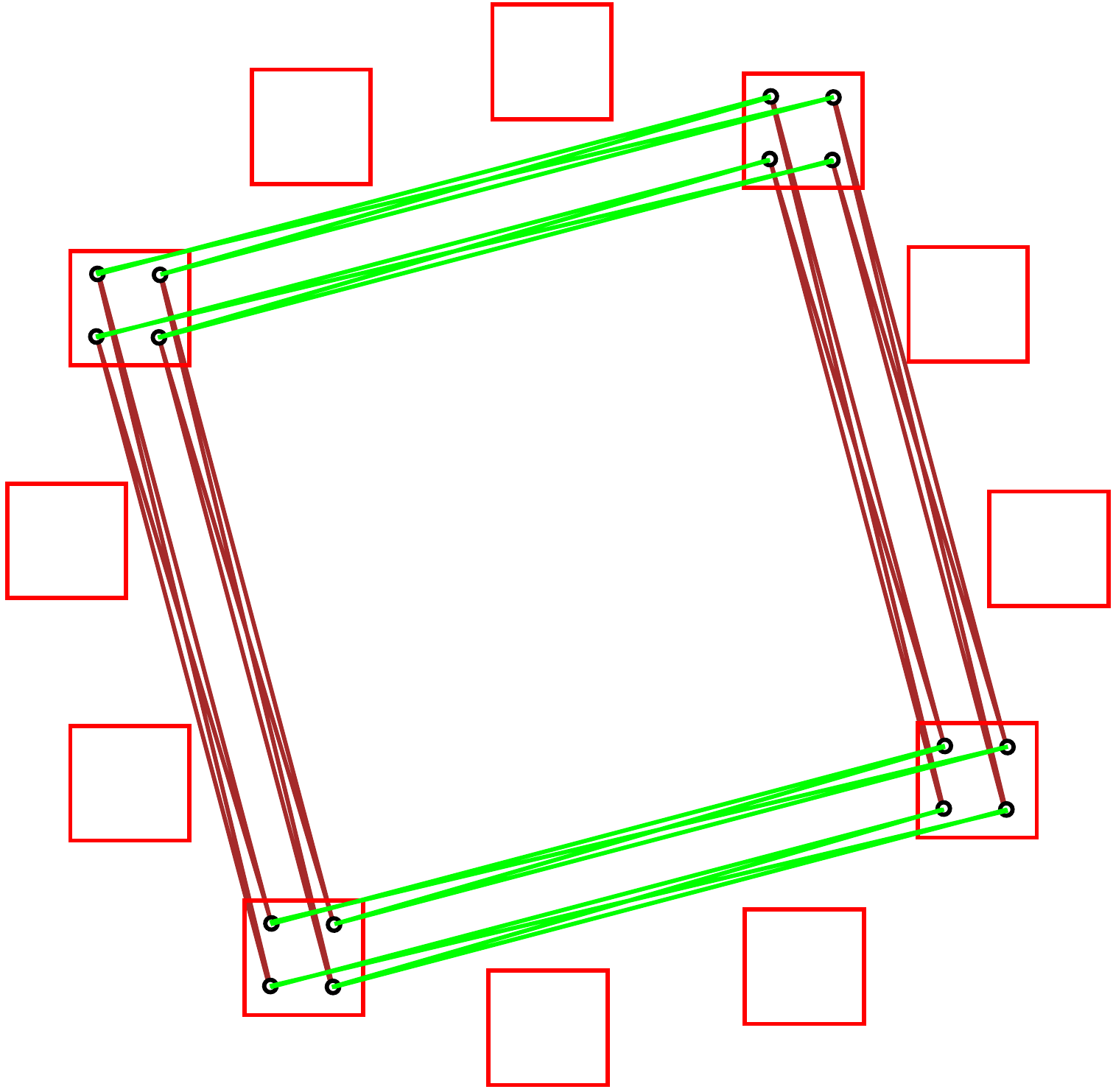,height=7cm}
}
\caption{Edges of $G_{2,12}$}
\end{center}
\end{figure}

\begin{claim}
$G_{n,m}$ is a vertex transitive graph. 
\end{claim}

\begin{proof}[{\bf Proof}]

Let $i(a,b),j(c,d)$ be two vertices of the graph $G_{n,m}$. In order to show that $G_{n,m}$ is vertex transitive, we need to construct an automorphism $\phi:V\to V$ that preserves edges and satisfies 
$\phi(i(a,b)) = j(c,d)$. 
We consider two separate cases, depending on whether $j - i$ is even or odd. 
\newline
If $j-i$ is even, then we set 
\[
\phi(k(u,v)) = ((k+ j - i) \bmod m) ((u+c-a)\bmod n, (v+d-b)\bmod n).
\]
If $j-i$ is odd, then we set
\[
\phi(k(u,v)) = ((k+ j - i) \bmod m) ((v+c-b)\bmod n, (u+d-a)\bmod n).
\]
It is straightforward to check that $\phi$ is an automorphism that preserves edges. 
\end{proof}

\begin{lemma}\label{lem:maxhit}
Let $X$ be a simple random walk on $G_{2,12}$ which is either discrete or continuous. Then we have
\begin{align}\label{eq:discreteeq}
\max_{x,y}\estart{\tau_y}{x} = \estart{\tau_{6(1,1)}}{0(0,0)}.
\end{align}
\end{lemma}

%

\begin{proof}[{\bf Proof}]

It suffices to prove the lemma for a discrete time random walk. Since $G_{2,12}$ is vertex transitive, it follows that for $x\in V$ we have
\[
\max_{x,y}\estart{\tau_y}{x} = \max_{y}\estart{\tau_y}{x}.
\]
So taking $x = 0(0,0)$, it suffices to show that for all $a,b \in \Z_2$ we have
\begin{align}\label{eq:maingoalproof}
\estart{\tau_{6(a,b)}}{0(0,0)} = \max_{y}\estart{\tau_y}{0(0,0)}.
\end{align}
First we observe that starting from any point in cluster 0, the first time the random walk hits cluster 6, the position is uniform.
Indeed, if we reach cluster 6 having used at least one short edge, then this is clear. If we use only long edges, then by the construction of the graph, with the first long edge we have randomized the column and with the second long edge we have randomized the row. 
Arguing similarly, if we start from cluster 0, the position at the first hitting time of cluster $i\neq 3,9$ is uniform.
Hence, if $T_i$ is the first time that we hit cluster $i\neq 3,9$, then 
\[
\estart{\tau_{i(a,b)}}{0(0,0)} = \estart{T_i}{0(0,0)} + \estart{\tau_{i(a,b)}}{U_i},
\]
where the last expectation means that we start from a uniform point in cluster $i$ and wait to hit $i(a,b)$. 
Since the graph is transitive, it follows that for all clusters $i$ and all $a,b$
\begin{align}\label{eq:constantexp}
\estart{\tau_{i(a,b)}}{U_i} =z.
\end{align}
We now define the process $Y$ to be the number of the cluster we are at. More precisely, $Y_t = i$ if and only if $X_t=i(a,b)$ for some $a,b$. It is easy to check that $Y$ is a Markov chain even with respect to the enlarged filtration which at time $t$ also contains the information about $X$ up to time $t$. The process $Y$ is a walk on $\Z_{12}$ with additional edges. 
From that it follows that for all $a,b$ we have
\[
\estart{T_i}{0(a,b)} = h(i) = \E{0\to i}
\]
and $h(i)$ satisfies a system of 
$6$ (by symmetry) linear equations, with solution given by
\begin{align}\label{eq:solution}
\nonumber&h(6) = 16, \ h(5) = h(7)= 16, \ h(4) = h(8)= 15, \\ &h(3)=h(9)= 13, \ h(2) = h(10) = 13, \ h(1) = h(11) = 10.
\end{align}

Putting everything together, we deduce that for all $i\neq 3,9$
\begin{align}\label{eq:hit6}
\estart{\tau_{i(a,b)}}{0(0,0)} = \E{0\to i} + \estart{\tau_{i(a,b)}}{U_i} = h(i) + z.
\end{align}
From~\eqref{eq:solution}, we obtain that
\begin{align}\label{eq:hardtarget}
\estart{\tau_{6(1,1)}}{0(0,0)} = \max_{\substack{i \neq 3,9 \\ a,b \in \Z_2}} \estart{\tau_{i(a,b)}}{0(0,0)}
\end{align}
and it remains to show that
\begin{align}\label{eq:finaltarget}
\estart{\tau_{6(1,1)}}{0(0,0)} \geq \max_{\substack{i=3,9 \\ a,b \in \Z_2}} \estart{\tau_{i(a,b)}}{0(0,0)}. 
\end{align}
Let $T$ be the first time that we hit cluster 3 without using the long edge $0\to 3$ directly. It then follows that at time $T$ the position in cluster $3$ is uniform. Hence we have
\[
\estart{\tau_{3(a,b)}}{0(0,0)} \leq \E{T} + \estart{\tau_{3(a,b)}}{U_3} = \E{T} + z.
\]
In view of~\eqref{eq:hit6} it thus suffices to show 
\begin{align}\label{eq:final}
\E{T} < \E{0\to 6} = 16.
\end{align}
Let $X$ be the first time that the walk is off the ``shuttle'' $0\to 3$. Then $X$ has the geometric distribution 
$\pr{X= i} = q p^{i-1}$ with $p = 1/6$ and $q=1-p = 5/6$. 
We can now write
\begin{align*}
\E{T} = 1 + \sum_{i=1,3,\ldots} \pr{X = i} A_1 + \sum_{i=2,4,\ldots} \pr{X= i} A_2,
\end{align*}
where $A_1$ and $A_2$ are given by 
\begin{align*}
& A_1 = \frac{2}{5} \E{1\to 3} + \frac{2}{5} \E{11\to 3} + \frac{1}{5} \E{9 \to 3} = \frac{72}{5} \\
& A_2 = \frac{2}{5} \E{4 \to 3} + \frac{2}{5} \E{2\to 3} + \frac{1}{5} \E{6\to 3} = \frac{53}{5}.
\end{align*}
Substituting  we deduce
\[
\E{T} = \frac{104}{7} < 16,
\]
and hence this concludes the proof of the lemma.
\end{proof}

\begin{proof}[{\bf Proof of Theorem~\ref{thm:example}} (for lazy walk)]
From Lemma~\ref{lem:maxhit} we have that the pair that maximizes $\estart{\tau_y}{x}$ is $x=0(0,0)$ and $y=6(1,1)$. (Lemma~\ref{lem:maxhit} is stated for a non-lazy walk, but the hitting times of the non-lazy and lazy walk are equal up to a factor of $2$.) 
We will now prove that if the moving target stays at position $5(1,1)$ for $2$ time steps and then moves to $6(1,1)$, then the expected hitting time is larger than $\estart{6(1,1)}{0(0,0)}$. 

We write $\tau_{5\to 6}$ for the time to hit the moving target. Then notice that $\tau_{5\to 6} - \tau_{6(1,1)}$ is non-zero if we hit $6$ at time $1$ or $2$.  We thus have
\begin{align*}
\estart{\tau_{5\to 6} - \tau_{6(1,1)}}{0(0,0)} \geq \prstart{\tau_{6(1,1)} \leq 2}{0(0,0)} \geq c >0
\end{align*}
and this concludes the proof of the theorem for a lazy walk.
\end{proof}

\begin{proof}[{\bf Proof of Theorem~\ref{thm:example}} (for continuous time walk)]

Consider the graph $G_{7,20}$. Solving the system of expected hitting times and arguing in exactly the same way as in the proof of Lemma~\ref{lem:maxhit} we get that
\[
\estart{\tau_{10(1,1)}}{0(0,0)} = \max_{x,y} \estart{\tau_y}{x}.
\]

We now describe a strategy for the moving particle that achieves bigger expected hitting time. Suppose that $f(t) = 8(1,1)$ when $t\leq \epsilon$ and $f(t) = 10(1,1)$ for $t>\epsilon$, where $
\epsilon>0$ will be determined. 

Note that $\tau_{f} - \tau_{10(1,1)}$ is nonzero if and only if $\tau_{10(1,1)}<\epsilon$ or $\tau_{f} <\epsilon$. To simplify notation we write $0$ instead of $0(0,0)$ and $\tau_{10}$ instead of $\tau_{10(1,1)}$.
We now have
\begin{align}\label{eq:decomp}
\nonumber\estart{\tau_f - \tau_{10}}{0} &= \estart{(\tau_f - \tau_{10})\1(\tau_f <\epsilon \ \text{ or } \ \tau_{10}<\epsilon)}{0} 
\\&=  \estart{(\tau_f - \tau_{10})\1(\tau_f <\epsilon)}{0} + \estart{(\tau_f - \tau_{10})\1(\tau_{10} <\epsilon, \tau_f>\epsilon)}{0}. 
\end{align}
We look at each of these two terms separately. For the first one we get
\begin{align}\label{eq:firstterm}
\estart{(\tau_f - \tau_{10})\1(\tau_f <\epsilon)}{0}  \geq \escond{\tau_f - \tau_{10}}{\tau_f <\epsilon, \tau_f <\tau_{10}}{0} \prstart{\tau_f <\epsilon, \tau_f <\tau_{10}}{0}. 
\end{align}
By the definition of $\tau_f$ we have $\{\tau_f < \epsilon\} = \{ \tau_{8(1,1)}<\epsilon\}$. We now describe an equivalent way of viewing the continuous time chain. To every edge adjacent to a vertex $x$ we assign an exponential clock of parameter $1/d(x)$, where $d(x)$ is the degree of $x$. 
Then the Markov chain crosses the edge of the first 
exponential clock that rings. In order to hit $8(1,1)$ before time $\epsilon$ at least four exponential clocks of a constant parameter should have rung. Thus 
\[
\prstart{\tau_f<\epsilon, \tau_{10}>\tau_f}{0}  \leq c \epsilon^4.
\]
It is easy to see that there exists a constant $c'$ independent of $\epsilon$ so that 
\[
\escond{\tau_{10}-\tau_f}{\tau_f<\epsilon,\tau_{10}>\tau_f}{0} \leq c'.
\]
Indeed, this expectation can be bounded from above by the commute time between $8(1,1)$ and~$10(1,1)$ which is at most twice the distance between $8(1,1)$ and $10(1,1)$ times the total number of edges of $G_{7,20}$.
Therefore plugging these estimates in~\eqref{eq:firstterm} we obtain for a positive constant $c_1$
\begin{align}\label{eq:firsttermbound}
\estart{(\tau_f - \tau_{10})\1(\tau_f <\epsilon)}{0} \geq -c_1 \epsilon^4.
\end{align}
For the second term of~\eqref{eq:decomp} we have
\begin{align*}
\estart{(\tau_f - \tau_{10})\1(\tau_{10} <\epsilon, \tau_f >\epsilon)}{0} \geq  \estart{(\tau_f - \tau_{10})\1(\tau_{10} <\epsilon/2, \tau_f >\epsilon)}{0} \geq \frac{\epsilon}{2} \prstart{\tau_{10}<\epsilon/2, \tau_f >\epsilon}{0}
\end{align*}
and arguing as above we obtain
\[
\prstart{\tau_{10}<\epsilon/2, \tau_f >\epsilon}{0} \asymp \epsilon^2.
\]
Putting all these estimates together we deduce
\[
\estart{\tau_f - \tau_{10}}{0} \geq c_2 \epsilon^3 - c_1\epsilon^4, 
\]
which can be made strictly positive by choosing $\epsilon>0$ sufficiently small and this completes the proof of the theorem.

\end{proof}

\section*{Acknowledgements}
We are grateful to Guy Kindler for proposing the question that led to this work. 
We thank Daniel Ahlberg, Almut Burchard, Yuval Peres, Richard Pymar and Alexandre Stauffer for useful discussions. We also thank MSRI, Berkeley, for its hospitality.

\bibliographystyle{plain}
\bibliography{biblio}

\begin{thebibliography}{1}

\bibitem{AldFill}
David Aldous and J.~Fill.
\newblock {\em Reversible Markov Chains and Random Walks on Graphs}.
\newblock In preparation,
  http://www.stat.berkeley.edu/$\sim$aldous/RWG/book.html.

\bibitem{BurSch}
A.~Burchard and M.~Schmuckenschl{\"a}ger.
\newblock Comparison theorems for exit times.
\newblock {\em Geom. Funct. Anal.}, 11(4):651--692, 2001.

\bibitem{SimonRoberto}
S.~{Griffiths}, R.~J. {Kang}, R.~{Imbuzeiro Oliveira}, and V.~{Patel}.
\newblock {Tight inequalities among set hitting times in Markov chains}.
\newblock {\em ArXiv e-prints}, August 2012.

\bibitem{Imbuzeiro}
R.~{Imbuzeiro Oliveira}.
\newblock {Mixing and hitting times for finite Markov chains}.
\newblock {\em ArXiv e-prints}, August 2011.

\bibitem{LevPerWil}
David~A. Levin, Yuval Peres, and Elizabeth~L. Wilmer.
\newblock {\em Markov chains and mixing times}.
\newblock American Mathematical Society, Providence, RI, 2009.
\newblock With a chapter by James G. Propp and David B. Wilson.

\bibitem{Oliveira}
Roberto~Imbuzeiro Oliveira.
\newblock On the coalescence time of reversible random walks.
\newblock {\em Trans. Amer. Math. Soc.}, 364(4):2109--2128, 2012.

\bibitem{PerSous11}
Y.~Peres and P.~Sousi.
\newblock An isoperimetric inequality for the {W}iener sausage.
\newblock {to appear in \it{GAFA}}.

\bibitem{PeresSousi}
Y.~{Peres} and P.~{Sousi}.
\newblock {Mixing times are hitting times of large sets}.
\newblock {\em ArXiv e-prints}, July 2011.

\end{thebibliography}

\end{document}